\documentclass[reqno]{amsart}
\usepackage{amsmath}
\usepackage{latexsym}
\usepackage{amsfonts}
\usepackage{amssymb}
\usepackage{amsfonts,euscript}
\usepackage{graphicx}
\DeclareGraphicsExtensions{.eps,epsfig,.bmp,.jpg,.pdf,.mps,.png,.gif}

\theoremstyle{plain}
\newtheorem{theorem}{\bf Theorem}[section]

\newtheorem{lemma}{\bf Lemma}[section]
\theoremstyle{definition}

\theoremstyle{remark}

\numberwithin{equation}{section}

\title[Stabilization of parabolic semilinear equations]{ Stabilization of Parabolic Semilinear Equations}

\author{Ionu\c{t} Munteanu}
\address{\noindent Ionu\c{t} Munteanu\newline Alexandru Ioan Cuza University, Department of Mathematics, and Octav Mayer Institute of Mathematics (Romanian Academy) 
\newline  700506 Ia\c{s}i, Romania}
 \email{ionut.munteanu@uaic.ro}

\begin{document}

\maketitle
\begin{abstract}We design here a finite-dimensional feedback stabilizing Dirichlet boundary controller for the equilibrium solutions to parabolic equations. These results extend  that ones in \cite{barbu1}, which provide a feedback controller expressed in terms of the eigenfunctions $\phi_j$ corresponding to the unstable eigenvalues $\left\{\lambda_j\right\}_{j=1}^N$ of the operator corresponding to the linearized equation. In \cite{barbu1}, the stabilizability result is conditioned by the require of   linear independence of $\left\{\frac{\partial}{\partial\nu}\phi_j\right\}_{j=1}^N$, on the part of the boundary where control acts. In this work, we design a similar control as in \cite{barbu1}, and show that it assures the stability of the system. This time, we drop the  require of linear independence and any other additional hypothesis.  Some examples are provided in order to illustrate the acquired results. More exactly, boundary stabilization of the heat equation and  the Fitzhugh-Nagumo equation. 
\end{abstract}
\noindent\textbf{Keywords:} semilinear parabolic equations, feedback controller, eigenvalue.

\noindent \textbf{2000 MSC:} 35K91; 93D15; 80A20.
\section{introduction}

We consider the following parabolic boundary stabilization problem in a  bounded domain $\Omega\subset\mathbb{R}^d,$ with a smooth boundary $\partial\Omega$, split in two, i.e., $\partial\Omega=\Gamma_1\cup\Gamma_2$, such that the Lebesgue measure $\sigma(\Gamma_1)\neq0:$
\begin{equation}\label{e40}\left\{\begin{array}{l}p_t(t,x)-\Delta p(t,x)+f(x,p)=0,\ t>0,\ x\in\Omega,\\
p(t,x)=u(t,x),\ t>0,\ x\in\Gamma_1,\\
\frac{\partial}{\partial\nu}p(t,x)=0,\ t>0,\ x\in \Gamma_2,\\
p(0,x)=p_o(x),\ x\in\Omega.
 \end{array}\right.\ \end{equation} Here,  $\nu$ stands for the unit outward   normal on the boundary $\partial\Omega$; $p_o\in L^2(\Omega)$ is  given (the initial data); $f$ is a nonlinear function for which, depending on the context, we choose from the  following two  hypotheses
\begin{itemize}
\item[$(f_1)$] $f,f_p\in C(\overline{\Omega}\times\mathbb{R})$;
\item[$(f_2)$] $f,f_p,f_{pp}\in C(\overline{\Omega}\times\mathbb{R})$, and there exist $C_1>0,\ q\in\mathbb{N},\  \alpha_i>0,\ i=1,...,q,$ when $d=1,2$, and $0< \alpha_i\leq 1,\ i=1,...,q, $ when $d=3$, such that $$|f_{pp}(x,p)|\leq C_1\left(\sum_{i=1}^q|p|^{\alpha_i}+1\right),\ \forall p\in \mathbb{R}.$$
\end{itemize} Here, $f_p,f_{pp}$ denote the partial derivatives of $f$ with respect to its second argument. On the part of the boundary $\Gamma_1$ it is applied the control $u$, while $\Gamma_2$ is insulated. 
 
 Let $p_e\in C^1(\overline{\Omega})$ be an equilibrium solution to (\ref{e40}), that is, $p_e$ satisfies
 $$-\Delta p_e+f(x,p_e)=0 \text{ in }\Omega;\ p_e=0 \text{ on }\Gamma_1,\ \frac{\partial}{\partial \nu}p_e=0 \text{ on }\Gamma_2.$$ Then, defining the fluctuation variable $w:=p-p_e$, equation (\ref{e40}) can be rewritten as
 \begin{equation}\label{e50}\left\{\begin{array}{l}w_t(t,x)-\Delta w(t,x)+f(x,w+p_e)-f(x,p_e)=0,\ t>0,\ x\in\Omega,\\
w(t,x)=\textbf{v}:=u(t,x)-p_e(x),\ t>0,\ x\in\Gamma_1,\\
\frac{\partial}{\partial\nu}w(t,x)=0,\ t>0,\ x\in \Gamma_2,\\
w(0,x)=w_o(x):=p_o(x)-p_e(x),\ x\in\Omega.
 \end{array}\right.\ \end{equation} 
 
 Our main concern here is the design of a boundary feedback controller $u$ that stabilizes exponentially the equilibrium state $p_e$ in (\ref{e40}), or, equivalently, the zero solution in (\ref{e50}).

 The problem of boundary feedback stabilization of this kind of parabolic-like equations was first solved in the pioneering work \cite{trig}. Then,  several others methods were proposed for deriving new types of controls, like backstepping approach, see \cite{bal,kk1,k1,k2}, or the proportional type controllers designed in \cite{barbu1,barbu2, ionut}. The last ones  provide a simple form feedback stabilizer expressed in terms of the unstable 
eigenfunctions of the operator given by the linearized equation around the equilibrium $p_e$  (that is the  operator $\mathcal{A}$ below), which is easily implementable in practice. However, the stabilizing procedure is conditioned  by the require that the  system $\left\{\frac{\partial}{\partial\nu}\phi_i\right\}_{i=1}^N$ is linearly independent in $L^2(\Gamma_1)$, where, $\phi_i,\ i=1,..., N$, are the corresponding first $N$ eigenfunctions.  In the present work, we  extend these results in the sense that we  drop the hypothesis of linear independence (which, in fact, seems to be quite restrictive, as shown in \cite{barbu1}) . More precisely, we design a similar type of control  as that one in \cite{barbu1}, and show that it assures the stability of the steady-state $p_e$, without needing the linear independence assumption.  Roughly saying,  the idea is as follows: in the case when the system $\left\{\frac{\partial}{\partial\nu}\phi_i\right\}_{i=1}^N$  is not linearly independent, the corresponding Gram matrix is singular, and, consequently, the approach in \cite{barbu1} fails because the controller  is not well-defined anymore. However, we overcome this  in the next manner: for the beginning, in order to clarify the procedure, we  work under the assumption that the first $N\in\mathbb{N}$ unstable eigenvalues are simple. We involve $N$ matrices obtained from the Gramian, multiplied by some  diagonal matrices, and show that  their sum is invertible, even if each one  is not. Therefore, we may well-define $u$ as the sum of  $N$ different controllers obtained from those matrices, and arrive to similar stabilization results as in \cite{barbu1}. The case of general eigenvalues involves a similar boundary feedback, as the one described above. The idea is to artificially make the eigenvalues simple again by slightly changing the linear operator.
 Finally, via a fixed point argument, locally stabilization results may be deduced for the fully nonlinear system (\ref{e40}).

 On this subject, it should be mentioned also the work \cite{badra}, where controllers, of the above type, are designed. Instead of the eigenfunctions system, they show that there exists a (possible different) family of functions for which, the corresponding proportional controller assures the stability. That result is based on some unique continuation property for the involved operators. Also, the work \cite{lasi}, on the Navier-Stokes equations, provide similar proportional feedbacks, but with the coefficients given by Riccati algebraic equations. However, in comparison to those results, here, the form of the feedback is given explicitly, and no additional assumption to the classical context is needed.

 \section{The design of the feedback law and the main results}
In this section, we detail the design of the new feedback, show how it differs from that one in \cite{barbu1}, and state the main stabilization results. To this end, let us define the operator $\mathcal{A}:\mathcal{D}(\mathcal{A})\rightarrow L^2(\Omega)$ as
 $$\mathcal{A}y:=-\Delta y+a(x)y,\ \forall y\in \mathcal{D}(\mathcal{A})=\left\{y\in H^2(\Omega): y|_{\Gamma_1}=0,\ \frac{\partial}{\partial\nu}y|_{\Gamma_2}=0\right\},$$ where $a(x):=f_y(x,y_e(x)),\ x\in \Omega$.  Notice that $\mathcal{A}$ is self-adjoint, besides this, it can be shown that  $\mathcal{A}$ has compact resolvent,  therefore, it has a countable set of eigenvalues, denoted by $\left\{\lambda_j\right\}_{j=1}^\infty $(repeated accordingly to their multiplicity) such that, given $\rho>0$,  there exists only a finite number  of eigenvalues $\left\{\lambda_i\right\}_{i=1}^N$ with $\lambda_i<\rho,\ i=1,...,N$ (we call them the unstable eigenvalues). So,  $\lambda_i\geq\rho,$ for all $i\geq N+1$ (for more details, see \cite{barbu1}). As we  see below, the larger $\rho$ is, the faster the exponential decay of the solution is. (But, on the other hand, the larger  $\rho$ is, the larger dimension, $N$, of the feedback is.)  Let us denote by $\left\{\phi_i\right\}_{i=1}^\infty$ the corresponding eigenfunctions, that can be chosen in such a way to form an orthonormal basis of the space $\mathcal{D}(\mathcal{A})$. 
 \subsection{The case of simple eigenvalues}
   For the beginning,  in order to make clearer our approach, we add to the above context the following hypothesis:
  $$(H)\ \ \ \text{ The eigenvalues $\lambda_i,\ i=1,...,N$, are mutually distinct}.$$ In other words, we assume that the first $N$ eigenvalues are simple. So, we can (eventually) rearrange them such that
  $$\lambda_1<\lambda_2<...<\lambda_N.$$ Then, in the subsection below, we show how we may drop this assumption, in order to obtain stabilization results in the framework of general eigenvalues.  
  
One of the tricks used here is to "lift" the boundary conditions into the equations. For this, we introduce the so-called Dirichlet operator as: given  $\alpha\in L^2(\Gamma_1)$ and $\gamma>0$, we denote by $D_\gamma \alpha:=y$, the solution to the equation
\begin{equation}\label{e4}\left\{\begin{array}{l}\displaystyle-\Delta y(x)+a(x)y(x)-2\sum_{k=1}^N\lambda_k\left<y,\phi_k\right>\phi_k(x)+\gamma y(x)=0,\text{ for } x\in\Omega,\\
y=\alpha, \text{ on } \Gamma_1,\\
\frac{\partial}{\partial\nu}y=0,\text{ on }\Gamma_2.\end{array}\right.\
\end{equation}
For $\gamma>0$ large enough equation (\ref{e4}) has a  solution, defining so the map $D_\gamma\in L(L^2(\Gamma_1),H^\frac{1}{2}(\Omega))$. Besides this, we have 
$$ \|D_\gamma\alpha\|_{H^\frac{1}{2}}\leq \frac{C}{\gamma}\|\alpha\|_{L^2(\Gamma_1)},\ \forall \alpha\in L^2(\Gamma_1).$$(see, e.g., \cite[p.6, line 16]{10}) We choose $\rho<\gamma_1<\gamma_2<...<\gamma_N$, $N$ constants such that equation (\ref{e4}) is well-posed for each of them, and denote by $D_{\gamma_i},\ i=1,...,N,$ the corresponding Dirichlet operator. 
  
  Next, let us denote by $\textbf{B}$ the Gram matrix of the system $\left\{\frac{\partial}{\partial\nu}\phi_i\right\}_{i=1}^N$ in the Hilbert space $L^2(\Gamma_1)$, with the standard scalar product  $\left<g,h\right>_0:=\int_{\Gamma_1}f(x)g(x)d\sigma$ ($\sigma$ being  the corresponding Lebesgue measure on the boundary $\Gamma_1$). That is
  \begin{equation}\label{e9}\textbf{B}:=\left(\begin{array}{cccc}\left<\frac{\partial}{\partial \nu}\phi_1,\frac{\partial}{\partial\nu}\phi_1\right>_0& \left<\frac{\partial}{\partial \nu}\phi_1,\frac{\partial}{\partial\nu}\phi_2\right>_0&...&\left<\frac{\partial}{\partial \nu}\phi_1,\frac{\partial}{\partial\nu}\phi_N\right>_0\\
\left<\frac{\partial}{\partial \nu}\phi_2,\frac{\partial}{\partial\nu}\phi_1\right>_0& \left<\frac{\partial}{\partial \nu}\phi_2,\frac{\partial}{\partial\nu}\phi_2\right>_0&...&\left<\frac{\partial}{\partial \nu}\phi_2,\frac{\partial}{\partial\nu}\phi_N\right>_0\\
...&...&...&...\\
\left<\frac{\partial}{\partial \nu}\phi_N,\frac{\partial}{\partial\nu}\phi_1\right>_0&\left<\frac{\partial}{\partial \nu}\phi_N,\frac{\partial}{\partial\nu}\phi_2\right>_0&...&\left<\frac{\partial}{\partial \nu}\phi_N,\frac{\partial}{\partial\nu}\phi_N\right>_0\end{array}\right).\end{equation} Further, we introduce the matrices
\begin{equation}\Lambda_{\gamma_k}:=\left(\begin{array}{cccc}\frac{1}{\gamma_k-\lambda_1}&0&...&0\\
0&\frac{1}{\gamma_k-\lambda_2}&...&0
\\...&...&...&...\\
0&0&...&\frac{1}{\gamma_k-\lambda_N} \end{array}\right),\ k=1,...,N,
\end{equation} 
 \begin{equation}\label{T}T:=\left(\begin{array}{cccc}\frac{1}{\gamma_1-\lambda_1}\frac{\partial}{\partial\nu}\phi_1(x) & \frac{1}{\gamma_1-\lambda_2}\frac{\partial}{\partial\nu}\phi_2(x) &...& \frac{1}{\gamma_1-\lambda_N}\frac{\partial}{\partial\nu}\phi_N(x)\\
\frac{1}{\gamma_2-\lambda_1}\frac{\partial}{\partial\nu}\phi_1(x) & \frac{1}{\gamma_2-\lambda_2}\frac{\partial}{\partial\nu}\phi_2(x) &...& \frac{1}{\gamma_2-\lambda_N}\frac{\partial}{\partial\nu}\phi_N(x)\\
...&...&...&...\\
\frac{1}{\gamma_N-\lambda_1}\frac{\partial}{\partial\nu}\phi_1(x) & \frac{1}{\gamma_N-\lambda_2}\frac{\partial}{\partial\nu}\phi_2(x) &...& \frac{1}{\gamma_N-\lambda_N}\frac{\partial}{\partial\nu}\phi_N(x)\end{array}\right),\end{equation}and \begin{equation}\label{AA}A=(B_1+B_2+...+B_N)^{-1},\end{equation}where  
\begin{equation}\label{bo}B_k:=\Lambda_{\gamma_k}\textbf{B}\Lambda_{\gamma_k},\ k=1,...,N.\end{equation}

At this stage, we are able to give the form of the stabilizing feedback, but, before this let us take a pause and make some comments on the invertibility of the sum  $B_1+...+B_N$, which is, in fact, the key result of this paper. More exactly, since the system $\left\{\frac{\partial}{\partial\nu}\phi_i\right\}_{i=1}^N$  is not necessarily  linearly independent in $L^2(\Gamma_1)$, the gramian $\textbf{B}$ may be singular, and, therefore,  each $B_k, k=1,...,N,$ may be  singular. However, by Lemma \ref{l1} in the Appendix,  the sum $B_1+...+B_N$ is an invertible matrix, provided that hypothesis $(H)$ holds true. Therefore, the matrix $A$ is well-defined.

Now, let us introduce the feedbacks

\begin{equation}\label{e1}u_k(t,x)=\left<A\left(\begin{array}{c}\left<p(t)-p_e,\phi_1\right>\\
\left<p(t)-p_e,\phi_2\right>\\...\\\left<p(t)-p_e,\phi_N\right>\end{array}\right),\left(\begin{array}{c}\frac{1}{\gamma_k-\lambda_1}\frac{\partial}{\partial \nu}\phi_1(x)\\ \frac{1}{\gamma_k-\lambda_2}\frac{\partial}{\partial \nu}\phi_2(x)\\ ... \\\frac{1}{\gamma_k-\lambda_N}\frac{\partial}{\partial \nu}\phi_N(x)\end{array}\right)\right>_N+p_e,\ t\geq0,\ x\in \Gamma_1,
\end{equation} for $k=1,2,...,N$. Here $\left<\cdot,\cdot\right>$ denotes the standard scalar product in $L^2(\Omega)$; while $\left<\cdot,\cdot\right>_N$ denotes the standard scalar product  in $\mathbb{R}^N$. 

In the case where the Gram matrix $\textbf{B}$ is nonsingular, one may consider only one feedback from the above list, for example $u_1$ (since, in this case,  $B_1^{-1}$ exists), take $u=u_1$ and show that it assures the  stability  stated in the main results of the paper, Theorems \ref{T1}, \ref{T2} below. In this case, in fact, we stumble on  the feedback designed in \cite{barbu1}. 

When $\textbf{B}$ is not invertible anymore, we take $u$ to be the sum
$$u=u_1+u_2+...+u_N,$$ which, in a condensed form, can be written as 
 \begin{equation}\label{u}u= \left<T\ A\left(\begin{array}{c}\left<p(t)-p_e,\phi_1\right>\\
\left<p(t)-p_e,\phi_2\right>\\...\\\left<p(t)-p_e,\phi_N\right>\end{array}\right),\left(\begin{array}{c}1\\ 1\\ ... \\1\end{array}\right)\right>_N+p_e.\end{equation}

We claim that the above feedback assures the stability of the steady-state $p_e$ in (\ref{e40}). In what follows, we focus to prove this is indeed so. The main ingredient toward the proposed goal is the stabilization of the linearized equation  corresponding to (\ref{e50}), given by,
 \begin{equation}\label{e51}\left\{\begin{array}{l}y_t(t,x)-\Delta y(t,x)+a(x)y(t,x)=0,\ t>0,\ x\in\Omega,\\
y(t,x)=\textbf{v}(t,x),\ t>0,\ x\in\Gamma_1,\\
\frac{\partial}{\partial\nu}y(t,x)=0,\ t>0,\ x\in \Gamma_2,\\
y(0,x)=y_o,\ x\in\Omega.
 \end{array}\right.\ \end{equation}
 The following results amounts to saying that the feedback $u$ achieves global exponential stability in the linear system (\ref{e51}), and local exponential stability in (\ref{e40}). More precisely
 \begin{theorem}\label{T1}Assume that hypothesis $(H)$ holds and $f$ obeys $(f_1)$. The feedback $u$, given by (\ref{u}), exponentially stabilizes the linearized equation (\ref{e51}). More exactly,  the solution $y$ to the equation
 \begin{equation}\label{e501}\left\{\begin{array}{l}y_t(t,x)-\Delta y(t,x)+f_y(x,y_e(x))y(t,x)=0,\ t>0,\ x\in\Omega,\\
y(t,x)=\left<T\ A\left(\begin{array}{c}\left<y(t),\phi_1\right>\\
\left<y(t),\phi_2\right>\\...\\\left<y(t),\phi_N\right>\end{array}\right),\left(\begin{array}{c}1\\ 1\\ ... \\1\end{array}\right)\right>_N,\ t>0,\ x\in\Gamma_1,\\
\frac{\partial}{\partial\nu}y(t,x)=0,\ t>0,\ x\in \Gamma_2,\\
y(0,x)=y_o,\ x\in\Omega,
 \end{array}\right.\ \end{equation}
 satisfies the exponential decay \begin{equation}\label{mu}\|y(t)\|^2\leq Ce^{-\mu t}\|y_o\|^2,\ t\geq0,\end{equation}for a prescribed $\mu>0$, and a constant $C>0$. Here $T$ is introduced in relation (\ref{T}), while $A$ is introduced in relation (\ref{AA}). $\|\cdot\|$ denotes the norm in $L^2(\Omega)$.
\end{theorem}

\begin{theorem}\label{T2}Assume that hypotheses $(H),(f_2)$ hold true. Then, the solution to the closed-loop nonlinear equation
 \begin{equation}\label{e502}\left\{\begin{array}{l}p_t(t,x)-\Delta p(t,x)+f(x,p)=0,\ t>0,\ x\in\Omega,\\
p(t,x)=\left<T\ A\left(\begin{array}{c}\left<p(t)-p_e,\phi_1\right>\\
\left<p(t)-p_e,\phi_2\right>\\...\\\left<p(t)-p_e,\phi_N\right>\end{array}\right),\left(\begin{array}{c}1\\ 1\\ ... \\1\end{array}\right)\right>_N+p_e(x),\ t>0,\ x\in\Gamma_1,\\
\frac{\partial}{\partial\nu}p(t,x)=0,\ t>0,\ x\in \Gamma_2,\\
p(0,x)=p_o,\ x\in\Omega,
 \end{array}\right.\ \end{equation}
 satisfies the exponential decay $$\|p(t)-p_e\|^2\leq Ce^{-\mu t}\|p_o-p_e\|^2,\ t\geq0,$$ for a prescribed $\mu>0$, and a constant $C>0$, provided that $\|p_o-p_e\|$ is small enough. Here $T$ is introduced in relation (\ref{T}), while $A$ is introduced in relation (\ref{AA}).

\end{theorem}
The proofs of these two theorems will be given latter, in Section \ref{S}.
\subsection{The case of general eigenvalues}
As announced above, we  see that, slightly perturbing the linear operator $\mathcal{A}$, we  are able to show that the present approach works  for the case of general unstable eigenvalues (i.e., not necessarily simple).  Indeed, let us assume, for example, that the first unstable eigenvalue has its multiplicity equal to 2, i.e., $\lambda_1=\lambda_2$, and the others unstable eigenvalues are simple (other cases can be treated in a similar manner as below). 

This time, the Dirichlet operator is introduced as follows: for any $\alpha\in L^2(\Gamma_1)$, we define  $D_\gamma\alpha:=y$, where $y$ is solution to
\begin{equation}\label{e4!}\left\{\begin{array}{l}\displaystyle -\Delta y(x)+a(x)y(x)-2\sum_{k=1}^N\lambda_k\left<y,\phi_k\right>\phi_k(x)-\delta\left<y,\phi_1\right>\phi_1+\gamma y(x)=0,\ x\in\Omega,\\
y=\alpha \text{ on } \Gamma_1,\\
\frac{\partial y}{\partial\nu}=0 \text{ on } \Gamma_2,\end{array}\right.\ \end{equation} for some $\delta>0$. We choose $\rho<\gamma_1<\gamma_2:=\gamma_1+\frac{1}{N-1}<\gamma_3:=\gamma_1+\frac{1}{N-2}...<\gamma_N:=\gamma_1+1,$ with $\gamma_1$   large enough, such that, for  $\delta=\frac{1}{\gamma_1^{4}}$ equation (\ref{e4!}) is well-posed for each $\gamma_i,\ i=1,...,N,$ and such that $\lambda_1+\delta,\lambda_2,...,\lambda_N$ are mutually distinct.

Now, our feedback has the form  
\begin{equation}\label{u!}u= \left<T\ A\left(\begin{array}{c}\left<p(t)-p_e,\phi_1\right>\\
\left<p(t)-p_e,\phi_2\right>\\...\\\left<p(t)-p_e,\phi_N\right>\end{array}\right),\left(\begin{array}{c}1\\ 1\\ ... \\1\end{array}\right)\right>_N+p_e.\end{equation}   $T$ is defined as
\begin{equation}\label{T!}T:=\left(\begin{array}{cccc}\frac{1}{\gamma_1-\delta-\lambda_1}\frac{\partial}{\partial\nu}\phi_1(x) & \frac{1}{\gamma_1-\lambda_2}\frac{\partial}{\partial\nu}\phi_2(x) &...& \frac{1}{\gamma_1-\lambda_N}\frac{\partial}{\partial\nu}\phi_N(x)\\
\frac{1}{\gamma_2--\delta-\lambda_1}\frac{\partial}{\partial\nu}\phi_1(x) & \frac{1}{\gamma_2-\lambda_2}\frac{\partial}{\partial\nu}\phi_2(x) &...& \frac{1}{\gamma_2-\lambda_N}\frac{\partial}{\partial\nu}\phi_N(x)\\
...&...&...&...\\
\frac{1}{\gamma_N-\delta-\lambda_1}\frac{\partial}{\partial\nu}\phi_1(x) & \frac{1}{\gamma_N-\lambda_2}\frac{\partial}{\partial\nu}\phi_2(x) &...& \frac{1}{\gamma_N-\lambda_N}\frac{\partial}{\partial\nu}\phi_N(x)\end{array}\right)\end{equation} and \begin{equation}\label{AA!}A=(B_1+B_2+...+B_N)^{-1},\end{equation}where 
\begin{equation}\label{bo!} B_k:=\Lambda_{\gamma_k}\textbf{B}\Lambda_{\gamma_k},\ k=1,...,N,\end{equation} for
\begin{equation}\label{e9!}\textbf{B}:=\left(\begin{array}{cccc}\left<\frac{\partial}{\partial \nu}\phi_1,\frac{\partial}{\partial\nu}\phi_1\right>_0& \left<\frac{\partial}{\partial \nu}\phi_1,\frac{\partial}{\partial\nu}\phi_2\right>_0&...&\left<\frac{\partial}{\partial \nu}\phi_1,\frac{\partial}{\partial\nu}\phi_N\right>_0\\
\left<\frac{\partial}{\partial \nu}\phi_2,\frac{\partial}{\partial\nu}\phi_1\right>_0& \left<\frac{\partial}{\partial \nu}\phi_2,\frac{\partial}{\partial\nu}\phi_2\right>_0&...&\left<\frac{\partial}{\partial \nu}\phi_2,\frac{\partial}{\partial\nu}\phi_N\right>_0\\
...&...&...&...\\
\left<\frac{\partial}{\partial \nu}\phi_N,\frac{\partial}{\partial\nu}\phi\right>_0&\left<\frac{\partial}{\partial \nu}\phi_N,\frac{\partial}{\partial\nu}\phi_3\right>_0&...&\left<\frac{\partial}{\partial \nu}\phi_N,\frac{\partial}{\partial\nu}\phi_N\right>_0\end{array}\right)\end{equation}and
\begin{equation} \Lambda_{\gamma_k}:=\left(\begin{array}{cccc}\frac{1}{\gamma_k-\delta-\lambda_1}&0&...&0\\
0&\frac{1}{\gamma_k-\lambda_2}&...&0
\\...&...&...&...\\
0&0&...&\frac{1}{\gamma_k-\lambda_N} \end{array}\right),\ k=1,...,N.\end{equation}
By the above definitions, taking into account that $\lambda_1+\delta,\lambda_2,...,\lambda_N$ are mutually distinct and $\gamma_1,\gamma_2,...,\gamma_N$ are also mutually distinct, one may show that the matrix $A$ is well-defined, that is, the sum $B_1+B_2+...+B_N$ is indeed invertible (this follows similarly as in Lemma \ref{l2} and Lemma \ref{l1} in the Appendix).

The main results, in the context of general eigenvalues, are stated below.
\begin{theorem}\label{T1!}Assume that $f$ satisfies $(f_1)$. Then, the feedback $u$, given by (\ref{u!}), exponentially stabilizes the linearized system (\ref{e51}). More precisely,  the solution $y$ to the system
 \begin{equation}\label{e501!}\left\{\begin{array}{l}y_t(t,x)-\Delta y(t,x)+f_y(x,y_e(x))y(t,x)=0,\ t>0,\ x\in\Omega,\\
y(t,x)=\left<T\ A\left(\begin{array}{c}\left<y(t),\phi_1\right>\\
\left<y(t),\phi_2\right>\\...\\\left<y(t),\phi_N\right>\end{array}\right),\left(\begin{array}{c}1\\ 1\\ ... \\1\end{array}\right)\right>_{N},\ t>0,\ x\in\Gamma_1,\\
\frac{\partial}{\partial\nu}y(t,x)=0,\ t>0,\ x\in \Gamma_2,\\
y(0,x)=y_o,\ x\in\Omega,
 \end{array}\right.\ \end{equation}
 satisfies the exponential decay \begin{equation}\label{mu!}\|y(t)\|^2\leq Ce^{-\mu t}\|y_o\|^2,\ t\geq0,\end{equation}for a prescribed $\mu>0$, and a constant $C>0$. Here $T$ is introduced in relation (\ref{T!}), while $A$ is introduced in relation (\ref{AA!}). $\|\cdot\|$ denotes the norm in $L^2(\Omega)$.
\end{theorem}

And, concerning the nonlinear equation.
\begin{theorem}\label{T2!} Let $1\leq d\leq 3$ and $f$ obeying assumption $(f_2)$. The solution to the closed-loop nonlinear system
 \begin{equation}\left\{\begin{array}{l}p_t(t,x)-\Delta p(t,x)+f(x,p)=0,\ t>0,\ x\in\Omega,\\
p(t,x)=\left<T\ A\left(\begin{array}{c}\left<p(t)-p_e,\phi_1\right>\\
\left<p(t)-p_e,\phi_2\right>\\...\\\left<p(t)-p_e,\phi_N\right>\end{array}\right),\left(\begin{array}{c}1\\ 1\\ ... \\1\end{array}\right)\right>_{N}+p_e(x),\ t>0,\ x\in\Gamma_1,\\
\frac{\partial}{\partial\nu}p(t,x)=0,\ t>0,\ x\in \Gamma_2,\\
p(0,x)=p_o,\ x\in\Omega,
 \end{array}\right.\ \end{equation}
 satisfies the exponential decay $$\|p(t)-p_e\|^2\leq Ce^{-\mu t}\|p_o-p_e\|^2,\ t\geq0,$$ for a prescribed $\mu>0$, and a constant $C>0$, provided that $\|p_o-p_e\|$ is small enough. Here $T$ is introduced in relation (\ref{T!}), while $A$ is introduced in relation (\ref{AA!}).
\end{theorem}
The proofs will be given latter, in Section \ref{S}.
\section{Examples}
In order to illustrate the acquired results, let us consider some examples arising from thermodynamics and biology. These examples are treated briefly, since we do not want to oversize the present work. However,  more details and also   numerical simulations can be found in the work \cite{liu} on the stabilization of the Fischer's equation by similar proportional feedbacks as above.

\textit{Example 1.} We consider first a problem discussed in \cite{bal}, and also in \cite{barbu1}. More precisely, the stabilization of the heat equation on the rod $(0,1)$
\begin{equation}\label{e60}\left\{\begin{array}{l}y_t-y_{xx}-\lambda y,\ x\in (0,1),\ t>0,\\
y_x(t,0)=0,\ y(t,1)=u(t),\ t>0,\end{array}\right.\
\end{equation}with the Dirichlet actuation $u$ in $x=1$ and with $\lambda$ a positive constant parameter. First of all, we notice that the stabilizing results using the backsteppting method in \cite{bal} work for arbitrarly level of instability, while in \cite{barbu1}  they are applicable under the condition that the parameter $\lambda$ do not exceed the bound $\left(\frac{3\pi}{2}\right)^2$.  In the present paper, Theorem \ref{T1} above holds true without any a priori condition on $\lambda$. This is indeed so, because, in our case, the operator $\mathcal{A}y=-y''-\lambda y,\ \forall y\in\mathcal{D}(\mathcal{A})=\left\{y\in H^2(0,1): \ y'(0)=y(1)=0\right\}$ has the eigenvalues $\lambda_j=\frac{(2j-1)^2\pi^2}{4}-\lambda$ with the corresponding eigenfunctions $\phi_j=\cos\frac{(2j-1)\pi}{2} x,\ j=1,2,...$. It is clear that hypothesis (H) holds true for any $N\in\mathbb{N}$. Let us pick some $N$ such that $\lambda_{j}>\rho>0,\ \forall j=N+1,N+2,....$ Since system (\ref{e60}) is linear and Theorem \ref{T1} is applicable we get that the feedback
$$u=\left<TA\left(\begin{array}{c}\int_0^1y(x)\cos\frac{\pi}{2}xdx\\ 
\int_0^1y(x)\cos\frac{3\pi}{2}xdx\\
...\\
\int_0^1y(x)\cos\frac{(2N-1)\pi}{2}xdx\end{array}\right),\left(\begin{array}{c}1\\ 1\\ ... \\ 1\end{array}\right)\right>_N,$$ where
$$T=\left(\begin{array}{cccc}-\frac{2\pi}{l_1-\pi^2} & \frac{6\pi}{l_1-(3\pi)^2} &...& (-1)^N\frac{2(2N-1)\pi}{l_1-[(2N-1)\pi]^2}\\
-\frac{2\pi}{l_2-\pi^2} & \frac{6\pi}{l_2-(3\pi)^2} &...& (-1)^N\frac{2(2N-1)\pi}{l_2-[(2N-1)\pi]^2}\\
...&...&...&...\\
-\frac{2\pi}{l_N-\pi^2} & \frac{6\pi}{l_N-(3\pi)^2} &...& (-1)^N\frac{2(2N-1)\pi}{l_N-[(2N-1)\pi]^2}\end{array}\right),$$ and
$$\begin{aligned}&A=\\& 
\displaystyle= \left\{\sum_{k=1}^N4\pi^2 \left(\begin{array}{cccc}\frac{1}{(l_k-\pi^2)^2}&  \frac{-3}{[l_k-\pi^2][l_k-(3\pi)^2]}&...& \frac{(-1)^{N+1}(2N-1)}{[l_k-\pi^2][l_k-(2N-1)^2\pi^2]}\\
\frac{-3}{[l_k-\pi^2][l_k-(3\pi)^2]}&\frac{9}{[l_k-(3\pi)^2]^2}&...&\frac{(-1)^N3(2N-1)}{[l_k-(3\pi)^2][l_k-(2N-1)^2\pi^2]}\\
...&...&...&...\\
\frac{(-1)^{N+1}(2N-1)}{[l_k-\pi^2][l_k-(2N-1)^2\pi^2]}& -\frac{(-1)^{N+1}3(2N-1)}{[l_k-(3\pi)^2][l_k-(2N-1)^2\pi^2]}
&...&\frac{(2N-1)^2}{[l_k-(2N-1)^2\pi^2]^2}\end{array} \right)\right\}^{-1},\end{aligned}$$ (here we have denoted by $l_k:=4(\gamma_k+\lambda),\ k=1,...,N$), assures the global exponential stability of the null solution of (\ref{e60}). Here, $\rho<\gamma_1<\gamma_2<...<\gamma_N$ are any positive constants.

\textit{Example 2.} We consider now the classical Fitzhugh-Nagumo equation that describes the dynamics of electrical potential across cell membrane (for details, see \cite{fiz})
\begin{equation}\label{e62}\left\{\begin{array}{l} y_t-y_{xx}+y(y-1)(y-a)=0,\ 0<x<l,\ t>0,\\
y_x(t,l)=0,\ y(t,0)=u(t),\ t>0,\end{array}\right.\ \end{equation}where $0<a<\frac{1}{2}$. It is known that the equilibrium $y_e=a$ is unstable. The linearized operator $\mathcal{A}$ is given by
$$\mathcal{A}y=-y''-a(1-a)y,\ \forall y\in\mathcal{D}(\mathcal{A})=\left\{y\in H^2(0,l):\ y(0)=y'(l)=0\right\}.$$It has the simple eigenvalues $\lambda_j=\frac{(2j-1)^2\pi^2}{4l^2}-a(1-a),\ j=1,2,...,$ with the corresponding eigenfunctions $\phi_j=\sin\frac{(2j-1)\pi}{2l}x,\ j=1,2,....$ Since hypothesis (H) is full filed for any $N\in\mathbb{N}$, we chose one  such that $\lambda_j>\rho>0,\ j=N+1,N+2,...$. It is easy to see that $f(y)=y(y-1)(y-a)$ obeys the hypothesis of Theorem \ref{T2}, and so we get that the feedback
$$u=\left<TA\left(\begin{array}{c}\int_0^l(y(x)-a)\sin\frac{\pi}{2l}xdx\\ 
\int_0^l(y(x)-a)\sin\frac{3\pi}{2l}xdx\\
...\\
\int_0^l(y(x)-a)\sin\frac{(2N-1)\pi}{2l}xdx\end{array}\right),\left(\begin{array}{c}1\\ 1\\ ... \\ 1\end{array}\right)\right>_N+a,$$ where
$$T=\left(\begin{array}{cccc}\frac{2l\pi}{l_1-\pi^2}&\frac{6l\pi}{l_1-(3\pi)^2}&...&\frac{2(2N-1)l}{l_1-(2N-1)^2\pi^2]}\\
\frac{2l\pi}{l_2-\pi^2}&\frac{6l\pi}{l_2-(3\pi)^2}&...&\frac{2(2N-1)l}{l_2-(2N-1)^2\pi^2]}\\
...&...&...&...\\
\frac{2l\pi}{l_N-\pi^2}&\frac{6l\pi}{l_N-(3\pi)^2}&...&\frac{2(2N-1)l}{l_N-(2N-1)^2\pi^2]}\end{array}\right),$$with $l_k:=4l^2(\gamma_k+a(1-a)),\ k=1,2,...,N;$ and
$$A=\left\{(2l)^2\pi^2\sum_{k=1}^N\left(\begin{array}{cccc}\left(\frac{1}{l_k-\pi^2}\right)^2& \frac{3}{[l_k-\pi^2][l_k-(3\pi)^2]}&...&\frac{2N-1}{[l_k-\pi^2][l_k-(2N-1)^2\pi^2]}\\
\frac{3}{[l_k-\pi^2][l_k-(3\pi)^2]}&\left( \frac{3}{l_k-(3\pi)^2}\right)^2&...& \frac{3(2N-1)}{[l_k-(3\pi)^2][l_k-(2N-1)^2\pi^2]}\\
...&...&...&...\\
\frac{2N-1}{[l_k-\pi^2][l_k-(2N-1)^2\pi^2]}& \frac{3(2N-1)}{[l_k-(3\pi)^2][l_k-(2N-1)^2\pi^2] }&... &\left(\frac{2N-1}{l_k-(2N-1)^2\pi^2}\right)^2
\end{array}\right)\right\}^{-1},$$ locally stabilizes the solution $y_e=a$ in (\ref{e62}).

\textit{Example 3.} Finally, we consider the periodic heat equation in $(0,\pi)^2$
\begin{equation}\label{e63}\left\{\begin{array}{l}y_t-\Delta y-\mu y=0,\ x\in (0,\pi)^2,\ t>0,\\
y(t,x_1,0)=u(t),\ y_{x_2}(t,x_1,\pi)=0,\ x_1\in(0,\pi),\\
 y_{x_1}(t,0,x_2)=y_{x_1}(t,\pi,x_2)=0,\ x_2\in (0,\pi).\end{array}\right.\ \end{equation}In this case the operator
 $$\mathcal{A}y=-\Delta y-\mu y,$$
 for all $$y\in\mathcal{D}(\mathcal{A})=\left\{y\in H^2((0,\pi)^2): y(x_1,0)=y_{x_2}(x_1,\pi)=0,\ y_{x_1}(0,x_2)=y_{x_1}(\pi,x_2)=0 \right\},$$ has the eigenvalues
 $$\lambda_k=k_1^2+\left(\frac{2k_2+1}{2}\right)^2-\mu,\ \forall k=(k_1,k_2)\in \mathbb{N}^2,$$ with the corresponding eigenfunctions
 $$\phi_k=\cos k_1x_1\sin\frac{2k_2+1}{2}x_2,\ \forall k\in\mathbb{N}^2.$$ Ordering the eigenvalues set as an increasing sequence and redefine them, we have
 $\lambda_1= 1.25-\mu,\ \phi_1=\cos x_1\sin\frac{1}{2}x_2;\ \lambda_2=3.25-\mu,\ \phi_2=\cos x_1\sin\frac{3}{2}x_2;\ \lambda_3=4.25-\mu,\ \phi_3=\cos2x_1\sin\frac{1}{2}x_2;\ \lambda_4=6.25-\mu,\ \phi_4=\cos2x_1\sin\frac{3}{2}x_2;\ \lambda_5=7.25-\mu,\phi_5=\cos x_1\sin\frac{5}{2}x_2;\ \lambda_6=9.25-\mu,\phi_6=\cos 3x_1\sin\frac{1}{2}x_2;\ \lambda_7=10.25-\mu,\phi_7=\cos2x_1\sin\frac{5}{2}x_2;\ \lambda_8=11.25-\mu,\phi_8=\cos3x_1\sin\frac{3}{2}x_2;\ \lambda_9=13.25-\mu,\phi_9=\cos x_1\sin \frac{7}{2}x_2;$ $\lambda_{10}=15.25-\mu,\phi_{10}=\cos3x_1\sin\frac{5}{2}x_2;\ \lambda_{11}=16.25-\mu,\phi_{11}=\cos4x_1\sin\frac{1}{2}x_2;\ \lambda_{12}=16.25-\mu,\phi_{12}=\cos2x_1\sin\frac{7}{2}x_2. $ It is clear that since $\lambda_{11}=\lambda_{12}$,  hypothesis $(H)$ fails to hold. So, we apply  this time the result in Theorem \ref{T1!}. Thus, the corresponding control $u$ of the form (\ref{u!})  assures the stability of the null solution in the system (\ref{e63}). We will not write it down here explicitly since it involves three  11th  order square matrices.

\section{Proofs of the results}\label{S}
\textbf{
\textit{Proof of Theorem \ref{T1}}.} The proof is based on similar ideas with those ones in the proof of \cite[Theorem 4.1]{barbu1}. We shall  show only the fact that the first $N$ modes of the solution $y$ are exponentially decaying. The rest  will be omitted because it can be deduced by almost identical arguments as in the proof of  \cite[Theorem 4.1]{barbu1}. The main difference between the present proof and that one in \cite{barbu1} is that while in \cite{barbu1}, the stability of the unstable modes is showed for each one independently, here we consider the  system formed by them, and show that it is  stable. That is why, the present feedback involves matrices, rather than vectors as it does in \cite{barbu1}.

Firstly, recall the Dirichlet operators $D_{\gamma_i},\ i=1,...,N,$ defined after relation (\ref{e4}). In the following, we  need to compute the scalar product $\left<D_{\gamma_j}\alpha,\phi_i\right>$, for all $i,j\in\left\{1,...,N\right\}$. To this end, we  scalarly multiply equation (\ref{e4}), corresponding to $\gamma_j$, by $\phi_i$, to get, via Green's formula, that
 \begin{equation}\label{e42}\begin{aligned}0&=-\int_\Omega \Delta y(x)\phi_i(x)dx+\int_\Omega a(x)y(x)\phi_i(x)dx+
 (\gamma_j-2\lambda_i)\int_\Omega y(x)\phi_i(x)dx\\&
 =\int_{\Gamma_1}\alpha(x)\frac{\partial}{\partial\nu}\phi_i(x)d\sigma+\int_\Omega y(x)(-\Delta\phi_i(x)+a(x)\phi_i(x))dx+(\gamma_j-2\lambda_i)\int_\Omega y(x)\phi_i(x)dx.\end{aligned}\end{equation}It yields that
\begin{equation}\label{e6}\begin{aligned}\left<D_{\gamma_j}\alpha,\phi_i\right>=-\frac{1}{\gamma_j-\lambda_i}\left<\alpha,\frac{\partial}{\partial\nu}\phi_i\right>_0,\ i,j=1,...,N.\end{aligned}\end{equation}

Next, we introduce the feedbacks
\begin{equation}v_k(t,x)=\left<A\left(\begin{array}{c}\left<y(t),\phi_1\right>\\
\left<y(t),\phi_2\right>\\...\\\left<y(t),\phi_N\right>\end{array}\right),\left(\begin{array}{c}\frac{1}{\gamma_k-\lambda_1}\frac{\partial}{\partial \nu}\phi_1(x)\\ \frac{1}{\gamma_k-\lambda_2}\frac{\partial}{\partial \nu}\phi_2(x)\\ ... \\\frac{1}{\gamma_k-\lambda_N}\frac{\partial}{\partial \nu}\phi_N(x)\end{array}\right)\right>_N,\ t\geq0,\ x\in \Gamma_1,
\end{equation} for $k=1,2,...,N$; and $\textbf{v}=v_1+v_2+...+v_N$. It is easy to see that we have
$$ \textbf{v}=\left<T\ A\left(\begin{array}{c}\left<y(t),\phi_1\right>\\
\left<y(t),\phi_2\right>\\...\\\left<y(t),\phi_N\right>\end{array}\right),\left(\begin{array}{c}1\\ 1\\ ... \\1\end{array}\right)\right>_N,$$ that is exactly the boundary feedback plugged in (\ref{e501}). Here, $T$ is defined in (\ref{T}), while $A$ is defined in (\ref{AA}).

For latter purpose, we show  that
\begin{equation}\label{e8}\left(\begin{array}{c}\left<D_{\gamma_k}v_k,\phi_1\right>\\ \left<D_{\gamma_k}v_k,\phi_2\right>\\...\\\left<D_{\gamma_k}v_k,\phi_N\right>\end{array}\right)=-B_kA\left(\begin{array}{c}\left<y(t),\phi_1\right>\\
\left<y(t),\phi_2\right>\\...\\\left<y(t),\phi_N\right>\end{array}\right),\end{equation} where  $B_k$ are introduced in (\ref{bo}) above, for $k=1,...,N$. This is indeed so. We have
$$\left<D_{\gamma_k}v_k,\phi_i\right>=\left<A \left(\begin{array}{c}\left<y(t),\phi_1\right>\\
\left<y(t),\phi_2\right>\\...\\\left<y(t),\phi_N\right>\end{array}\right),\left(\begin{array}{c}\frac{1}{\gamma_k-\lambda_1}\left<D_{\gamma_k}\frac{\partial}{\partial \nu}\phi_1,\phi_i\right>\\ \frac{1}{\gamma_k-\lambda_2}\left<D_{\gamma_k}\frac{\partial}{\partial \nu}\phi_2,\phi_i\right>\\ ... \\\frac{1}{\gamma_k-\lambda_N}\left<D_{\gamma_k}\frac{\partial}{\partial \nu}\phi_N,\phi_i\right>\end{array}\right)\right>_N,\ i=1,...,N.$$It then follows by relation (\ref{e6}),  that
\begin{equation}\label{e7}\left<D_{\gamma_k}v_k,\phi_i\right>=\left<A \left(\begin{array}{c}\left<y(t),\phi_1\right>\\
\left<y(t),\phi_2\right>\\...\\\left<y(t),\phi_N\right>\end{array}\right),\left(\begin{array}{c}-\frac{1}{(\gamma_k-\lambda_1)(\gamma_k-\lambda_i)}\left<\frac{\partial}{\partial \nu}\phi_1,\frac{\partial}{\partial\nu}\phi_i\right>_0\\ -\frac{1}{(\gamma_k-\lambda_2)(\gamma_k-\lambda_i)}\left<\frac{\partial}{\partial \nu}\phi_2,\frac{\partial}{\partial\nu}\phi_i\right>_0\\ ... \\-\frac{1}{(\gamma_k-\lambda_N)(\gamma_k-\lambda_i)}\left<\frac{\partial}{\partial \nu}\phi_N,\frac{\partial}{\partial\nu}\phi_i\right>_0\end{array}\right)\right>_N,\ i=1,...,N,\end{equation}from where we immediately obtain (\ref{e8}).

Now, let us return to the linear system
\begin{equation}\label{e11}\left\{\begin{array}{l}y_t(t,x)-\Delta y(t,x)+a(x)y(t,x)=0,\ t>0,\ x\in\Omega,\\
y(t,x)=\textbf{v}(t,x)=v_1(t,x)+...+v_N(t,x),\ t>0,\ x\in \Gamma_1,\\
\frac{\partial}{\partial\nu}y(t,x)=0,\ t>0,\ x\in \Gamma_2,\\
y(0,x)=y_o(x),\ x\in\Omega.\end{array} \right.\ \end{equation}
Let us denote by $z(t,x):=y(t,x)-D_{\gamma_1}v_1(t,x)-...-D_{\gamma_N}v_N(t,x),\ t\geq0,\ x\in\Omega$. We claim that the feedbacks $v_k,\ k=1,...,N,$  can be expressed only in terms of $z$, as
\begin{equation}\label{e13}v_k(t,x)=\frac{1}{2}\left<A\left(\begin{array}{c}\left<z(t),\phi_1\right>\\
\left<z(t),\phi_2\right>\\...\\\left<z(t),\phi_N\right>\end{array}\right),\left(\begin{array}{c}\frac{1}{\gamma_k-\lambda_1}\frac{\partial}{\partial \nu}\phi_1\\ \frac{1}{\gamma_k-\lambda_2}\frac{\partial}{\partial \nu}\phi_2\\ ... \\\frac{1}{\gamma_k-\lambda_N}\frac{\partial}{\partial \nu}\phi_N\end{array}\right)\right>_N.
\end{equation} To see this we do the following straightforward computations
$$\begin{aligned}&\frac{1}{2}\left<A\left(\begin{array}{c}\left<z(t),\phi_1\right>\\
\left<z(t),\phi_2\right>\\...\\\left<z(t),\phi_N\right>\end{array}\right),\left(\begin{array}{c}\frac{1}{\gamma_k-\lambda_1}\frac{\partial}{\partial \nu}\phi_1\\ \frac{1}{\gamma_k-\lambda_2}\frac{\partial}{\partial \nu}\phi_2\\ ... \\\frac{1}{\gamma_k-\lambda_N}\frac{\partial}{\partial \nu}\phi_N\end{array}\right)\right>_N=\frac{1}{2}\left<A\left(\begin{array}{c}\left<y(t),\phi_1\right>\\
\left<y(t),\phi_2\right>\\...\\\left<y(t),\phi_N\right>\end{array}\right),\left(\begin{array}{c}\frac{1}{\gamma_k-\lambda_1}\frac{\partial}{\partial \nu}\phi_1\\ \frac{1}{\gamma_k-\lambda_2}\frac{\partial}{\partial \nu}\phi_2\\ ... \\\frac{1}{\gamma_k-\lambda_N}\frac{\partial}{\partial \nu}\phi_N\end{array}\right)\right>_N\\&
-\frac{1}{2}\sum_{i=1}^N \left<A\left(\begin{array}{c}\left<D_{\gamma_i}v_i(t),\phi_1\right>\\
\left<D_{\gamma_i}v_i(t),\phi_2\right>\\...\\\left<D_{\gamma_i}v_i(t),\phi_N\right>\end{array}\right), \left(\begin{array}{c}\frac{1}{\gamma_k-\lambda_1}\frac{\partial}{\partial \nu}\phi_1\\ \frac{1}{\gamma_k-\lambda_2}\frac{\partial}{\partial \nu}\phi_2\\ ... \\ \frac{1}{\gamma_k-\lambda_N}\frac{\partial}{\partial \nu}\phi_N\end{array}\right)\right>_N\\&
\text{ (taking into account relation (\ref{e8}))}\\&=\frac{1}{2}\left<\left[I+A(B_1+...+B_ N)\right]A\left(\begin{array}{c}\left<y(t),\phi_1\right>\\
\left<y(t),\phi_2\right>\\...\\\left<y(t),\phi_N\right>\end{array}\right),\left(\begin{array}{c}\frac{1}{\gamma_k-\lambda_1}\frac{\partial}{\partial \nu}\phi_1\\ \frac{1}{\gamma_k-\lambda_2}\frac{\partial}{\partial \nu}\phi_2\\ ... \\\frac{1}{\gamma_k-\lambda_N}\frac{\partial}{\partial \nu}\phi_N\end{array}\right)\right>_N\\&
=v_k,
\end{aligned},$$ since $A=(B_1+...+B_N)^{-1}$. Moreover, likewise in (\ref{e8}), we have now
\begin{equation}\label{e17}\left(\begin{array}{c}\left<D_{\gamma_k}v_k,\phi_1\right>\\ \left<D_{\gamma_k}v_k,\phi_2\right>\\...\\\left<D_{\gamma_k}v_k,\phi_N\right>\end{array}\right)=-\frac{1}{2}B_kA\left(\begin{array}{c}\left<z(t),\phi_1\right>\\
\left<z(t),\phi_2\right>\\...\\\left<z(t),\phi_N\right>\end{array}\right),\ k=1,...,N.\end{equation} 

Finally, equation (\ref{e11}) may be rewritten in terms of $z$ as follows
\begin{equation}\label{e15}\left\{\begin{array}{l}\begin{aligned}&z_t(t,x)+\mathcal{A}z(t,x)=R(\left<z,\phi_1\right>,...,\left<z,\phi_N\right>),\ t>0,\ x\in\Omega,\end{aligned}\\
z(0,x)=z_o(x),\ x\in\Omega,\end{array} \right.\ \end{equation}
where
\begin{equation}\label{r}\displaystyle R(\left<z,\phi_1\right>,...,\left<z,\phi_N\right>):=-\left(\sum_{i=1}^ND_{\gamma_i}v_i\right)_t-2\sum_{i,j=1}^N\lambda_j\left<D_{\gamma_i}v_i,\phi_j\right>\phi_j+\sum_{i=1}^N\gamma_iD_{\gamma_i}v_i.\end{equation}By (\ref{e17}), and using the fact that $A$ is the inverse of the sum of $B_k,\ k=1,...,N$, we see immediately that
\begin{equation}\label{rr}\begin{aligned}\left(\begin{array}{c}\left<R,\phi_1\right>\\ \left<R,\phi_2\right>\\ ... \\ \left<R,\phi_N\right>\end{array}\right)&=\frac{1}{2}\sum_{k=1}^NB_kA\mathcal{Z}_t+\Lambda\sum_{k=1}^NB_kA\mathcal{Z}-\frac{1}{2}\sum_{k=1}^N\gamma_kB_kA\mathcal{Z}\\&
=\frac{1}{2}\mathcal{Z}_t+\Lambda\mathcal{Z}-\frac{1}{2}\gamma_1\mathcal{Z}+\frac{1}{2}\sum_{k=2}^N(\gamma_1-\gamma_k)B_kA\mathcal{Z},\end{aligned}
\end{equation}
where we have denoted by $\mathcal{Z}(t):=\left(\begin{array}{c}\left<z(t),\phi_1\right>\\
\left<z(t),\phi_2\right>\\...\\\left<z(t),\phi_N\right>\end{array}\right),\ t\geq0;$ and by $\Lambda:=\left(\begin{array}{cccc}\lambda_1&0&...&0\\
0&\lambda_2&...&0\\
...&...&...&...\\
0&0&...&\lambda_N\end{array}\right).$

We have now, by (\ref{e15}), scalarly multiplied by $\phi_i,\ i=1,...,N$, together with (\ref{rr}), that
\begin{equation}\label{e44e}\mathcal{Z}_t+\Lambda\mathcal{Z}=\frac{1}{2}\mathcal{Z}_t+\Lambda\mathcal{Z}-\frac{1}{2}\gamma_1\mathcal{Z}+\frac{1}{2}\sum_{k=2}^N(\gamma_1-\gamma_k)B_kA\mathcal{Z},\ t>0;\ \mathcal{Z}(0)=\mathcal{Z}_o.\end{equation}Or, equivalently,
\begin{equation}\label{e44ee}\mathcal{Z}_t=-\gamma_1\mathcal{Z}+\sum_{k=2}^N(\gamma_1-\gamma_k)B_kA\mathcal{Z},\ t>0;\ \mathcal{Z}(0)=\mathcal{Z}_o.\end{equation}
Recall that $B_j,\ j=1,...,N,$ are positive semidefinite symmetric matrices (by the definition of $B_j,\Lambda_{\gamma_j}$ and the fact that $\textbf{B}$ is a Gram matrix), therefore, $\left<B_jq,q\right>_N\geq0,\ \forall q\in\mathbb{R}^N,\ j=1,...,N.$ Consequently,   $A=\left(B_1+...+B_N\right)^{-1}$ is a positive definite symmetric matrix. Thus one can define another positive definite symmetric matrix, denoted by $A^\frac{1}{2}$, such that $A^\frac{1}{2}A^\frac{1}{2}=A$ (the square root of $A$; for details see \cite{bour}). Let us scalarly multiply equation (\ref{e44ee}) by $A\mathcal{Z}$, to get
\begin{equation}\label{e70}\begin{aligned}\frac{1}{2}\frac{d}{dt}\|A^\frac{1}{2}\mathcal{Z}(t)\|_N^2&=-\gamma_1\|A^\frac{1}{2}\mathcal{Z}(t)\|_N^2 +\sum_{k=2}^N(\gamma_1-\gamma_k)\left<B_kA\mathcal{Z}(t),A\mathcal{Z}(t)\right>_N,\end{aligned}\end{equation}
that leads to $$\frac{1}{2}\frac{d}{dt}\|A^\frac{1}{2}\mathcal{Z}(t)\|_N^2\leq-\gamma_1\|A^\frac{1}{2}\mathcal{Z}(t)\|_N^2,\ t\geq0,$$ since $\gamma_1-\gamma_k<0,\ k=2,...,N$. Here $\|\cdot\|_N$ stands for the euclidean norm in $\mathbb{R}^N$. The above relation implies the exponential decay of $\mathcal{Z}$ in the $\|A^\frac{1}{2}\cdot\|_N$-norm, i.e.,
$$\|A^\frac{1}{2}\mathcal{Z}(t)\|_N^2\leq e^{-2\gamma_1 t}\|A^\frac{1}{2}\mathcal{Z}_o\|_N^2,\ t\geq0,$$where using  the fact that $A^\frac{1}{2}$ is a positive definite symmetric matrix, we finally arrive to
\begin{equation}\label{e45}\|\mathcal{Z}(t)\|_N^2\leq Ce^{-2\gamma_1t}\|\mathcal{Z}_o\|_N^2,\ t\geq0,\end{equation}for some positive constant $C$.

The rest of the proof follows by identical arguments as in the proof of \cite[Theorem 4.1]{barbu1}, therefore it is omitted. However, we  notice that, as shown in \cite{barbu1}, the remaining infinitely many modes will be exponentially stabilized by the feedback $\textbf{v}$ with the gain $e^{-\rho t}$. Hence, by (\ref{e45}), since $\gamma_1$ and $\rho$ may be taken arbitrarily large, the decaying constant $\mu$ in (\ref{mu}) may be taken arbitrarily large as well.  $\square$

In the case of general eigenvalues, we have

\textbf{\textit{Proof of Theorem \ref{T1!}}} The proof is almost identical with that one of Theorem \ref{T1}, that is why we shall skip most of computational details. 

Recall the Dirichlet operators $D_{\gamma_i},\ i=1,...,N,$ introduced after relation (\ref{e4!}). Likewise in (\ref{e42})-(\ref{e6}) we get
\begin{equation}\label{e6!}\begin{array}{l}\left<D_{\gamma_j} \alpha,\phi_1\right>=-\frac{1}{\gamma_j-\delta-\lambda_1}\left<\alpha,\frac{\partial \phi_1}{\partial\nu}\right>_0 \text{ and }
\left<D_{\gamma_j} \alpha,\phi_i\right>=-\frac{1}{\gamma_j-\lambda_i}\left<\alpha,\frac{\partial \phi_i}{\partial\nu}\right>_0,\end{array}\end{equation}for all $ i=2,...,N,\ j=1,...,N.$ 

 Now, introduce the feedbacks
\begin{equation}v_k(t,x)=\left<A\left(\begin{array}{c}\left<y(t),\phi_1\right>\\
\left<y(t),\phi_2\right>\\...\\\left<y(t),\phi_N\right>\end{array}\right),\left(\begin{array}{c}\frac{1}{\gamma_k-\delta-\lambda_1}\frac{\partial}{\partial \nu}\phi_1(x)\\ \frac{1}{\gamma_k-\lambda_2}\frac{\partial}{\partial \nu}\phi_2(x)\\ ... \\\frac{1}{\gamma_k-\lambda_N}\frac{\partial}{\partial \nu}\phi_N(x)\end{array}\right)\right>_{N},\ t\geq0,\ x\in \Gamma_1,
\end{equation} for $k=1,...,N$; and $\textbf{v}=v_1+v_2+...+v_N$. It is clear that
$$ \textbf{v}=\left<T\ A\left(\begin{array}{c}\left<y(t),\phi_1\right>\\
\left<y(t),\phi_2\right>\\...\\\left<y(t),\phi_N\right>\end{array}\right),\left(\begin{array}{c}1\\ 1\\ ... \\1\end{array}\right)\right>_{N},$$ which is the boundary feedback plugged in (\ref{e501!}).

Set $z(t,x):=y(t,x)-D_{\gamma_1} v_1(t,x)-D_{\gamma_2}v_2(t,x)-...-D_{\gamma_N}v_N(t,x)$.   Similarly as in (\ref{e13}) we have 
 \begin{equation}\label{e13!!}v_k(t,x)=\frac{1}{2}\left<A\left(\begin{array}{c}\left<z(t),\phi_1\right>\\
\left<z(t),\phi_2\right>\\...\\\left<z(t),\phi_N\right>\end{array}\right),\left(\begin{array}{c}\frac{1}{\gamma_k-\delta-\lambda_1}\frac{\partial}{\partial \nu}\phi_1\\ \frac{1}{\gamma_k-\lambda_2}\frac{\partial}{\partial \nu}\phi_2\\ ... \\\frac{1}{\gamma_k-\lambda_N}\frac{\partial}{\partial \nu}\phi_N\end{array}\right)\right>_{N},\ k=1,...,N.
\end{equation}Besides this, as in (\ref{e17}), we also have that
\begin{equation}\label{e17!!}\left(\begin{array}{c}\left<D_{\gamma_k}v_k,\phi_1\right>\\ \left<D_{\gamma_k}v_k,\phi_2\right>\\...\\\left<D_{\gamma_k}v_k,\phi_N\right>\end{array}\right)=-\frac{1}{2}B_kA\left(\begin{array}{c}\left<z(t),\phi_1\right>\\
\left<z(t),\phi_2\right>\\...\\\left<z(t),\phi_N\right>\end{array}\right),\ k=1,...,N.\end{equation}

In terms of the new variable $z$, the linearization (\ref{e501!})   may be rewritten as
\begin{equation}\label{e15!}\left\{\begin{array}{l}\begin{aligned}&z_t(t,x)+\mathcal{A}z(t,x)=R(\left<z,\phi\right>,\left<z,\phi_3\right>,...,\left<z,\phi_N\right>),\ t>0,\ x\in\Omega,\end{aligned}\\
z(0,x)=z_o(x):=y_o-D_{\gamma_1} v_1(0,x)-D_{\gamma_2}v_2(0,x)-...-D_{\gamma_N}v_N(0,x),\ x\in\Omega,\end{array} \right.\ \end{equation}
where
\begin{equation}\label{r!}\begin{aligned}\displaystyle  R(\left<z,\phi_1\right>,\left<z,\phi_2\right>,...,\left<z,\phi_N\right>)&:=-\left(\sum_{i=1}^ND_{\gamma_i}v_i\right)_t
-2\sum_{i,j=1}^N\lambda_j\left<D_{\gamma_i}v_i,\phi_j\right>\phi_j\\&
-\delta\sum_{i=1}^N \left<D_{\gamma_i}v_i,\phi_1\right>\phi_1+\sum_{i=1}^N\gamma_iD_{\gamma_i}v_i.\end{aligned}\end{equation}Then, simple computations as in (\ref{rr}) give
\begin{equation}\label{rr!!}\begin{aligned}\left(\begin{array}{c}\left<R,\phi_1\right>\\ \left<R,\phi_2\right>\\ ... \\ \left<R,\phi_N\right>\end{array}\right)&=\frac{1}{2}\mathcal{Z}_t+\Lambda\mathcal{Z}-\frac{1}{2}\gamma_1\mathcal{Z}+\frac{1}{2}\sum_{k=2}^N(\gamma_1-\gamma_k)B_kA\mathcal{Z}-\frac{\delta}{2}\mathcal{O},\end{aligned}
\end{equation}
where we have denoted by $\mathcal{Z}(t):=\left(\begin{array}{c}\left<z(t),\phi_1\right>\\
\left<z(t),\phi_2\right>\\...\\\left<z(t),\phi_N\right>\end{array}\right),\ t\geq0;$  by $\Lambda:=\left(\begin{array}{cccc}\lambda_1&0&...&0\\
0&\lambda_2&...&0\\
...&...&...&...\\
0&0&...&\lambda_N\end{array}\right);$and by $ \mathcal{O}:=\left(\begin{array}{c}\left<z(t),\phi_1\right>\\ 0\\...\\0\end{array}\right).$ Then, likewise in (\ref{e70}), we get

\begin{equation}\label{mun1}\frac{d}{dt}\|A^\frac{1}{2}\mathcal{Z}(t)\|_N^2\leq -2\gamma_1\|A^\frac{1}{2}\mathcal{Z}(t)\|_N^2+\delta \|A\|\|\mathcal{Z}(t)\|_N^2,\ t\geq0,\end{equation}where $\|A\|$ stands for the classical euclidean norm of the matrix $A$. Denote by $\lambda_1(A)>0$ the first eigenvalue of $A$, then by integration with respect to time in (\ref{mun1}), it yields
\begin{equation}\label{mun2}\lambda_1(A)\|\mathcal{Z}(t)\|_N^2\leq e^{-2\gamma_1t}\|A^\frac{1}{2}\mathcal{Z}_0\|_N^2+\int_0^te^{-2\gamma_1(t-s)}\delta \|A\|\|\mathcal{Z}(s)\|_N^2ds,\end{equation}where making use of the Gronwall's lemma
\begin{equation}\label{mun3}\|\mathcal{Z}(t)\|_N^2\leq \frac{1}{\lambda_1(A)}\|A^\frac{1}{2}\mathcal{Z}_0\|_N^2\exp\left[\left(\frac{\delta\|A\|}{\lambda_1(A)}-2\gamma_1\right)t\right],\ t\geq0.\end{equation}

Let us denote by $b_{ij},\ i,j=1,...,N,$ the entries of the matrix $B_1+B_2+...+B_N$. By the definition of $B_i$ (see (\ref{bo!})) and of the constants $\gamma_i,\ i=1,...,N$ and $\delta$ (see after (\ref{e4!})), we have $\lim_{\gamma_1\rightarrow\infty}\gamma_1^2|b_{ij}|\in \mathbb{R}_+,\forall i,j=1,...,N.$ Set $b_{ij}^*$ the entries 
of the adjoint of the matrix $B_1+...+B_N$. By definition of the adjoint and the above observation we deduce that $\lim_{\gamma_1\rightarrow \infty}\gamma_1^{2(N-1)}|b_{ij}^*|\in\mathbb{R}_+,\ \forall i,j=1,...,N$. Besides this, the above observation  also implies that $\lim_{\gamma_1\rightarrow\infty}\gamma_1^{2(N+1)}|det(B_1+...+B_N)|=+\infty.$ In other words, we have
$$|b_{ij}^*|\leq c_{ij}\frac{1}{\gamma_1^{2(N-1)}},\ i,j=1,...,N, \text{ and }\left|\frac{1}{det(B_1+...+B_N)}\right|\leq c \gamma_1^{2(N+1)},$$for some positive constants $c_{ij}, c,\ i,j=1,...,N$, independent of $\gamma_1$, for $\gamma_1$ large enough. This yields that, denoting by $a_{ij},\ i,j=1,...,N,$ the entries of the matrix $A=(B_1+...+B_N)^{-1}$, there exists some constant $C>0$, independent of $\gamma_1$, such that $|a_{ij}|\leq C\gamma_1^4,\ i,j=1,...,N.$ Consequently
$$\|A\|\leq CN^2\gamma_1^4,$$for $\gamma_1$ large enough. In conclusion, for $\gamma_1$ large enough, there exists some $\mu>0$, such that
$$\frac{\delta \|A\|}{\lambda_1(A)}-2\gamma_1=\frac{ \|A\|}{\gamma_1^{4}\lambda_1(A)}-2\gamma_1\leq CN^2\frac{1}{\lambda_1(A)}-2\gamma_1\leq -\mu,$$ since $\frac{1}{\lambda_1(A)}\rightarrow 0$ for $\gamma_1\rightarrow \infty$. This, together with (\ref{mun3}), yields
$$\|\mathcal{Z}(t)\|_N^2\leq \frac{1}{\lambda_1(A)}e^{-\mu t}\|A^\frac{1}{2}\mathcal{Z}_0\|_N^2,\ t\geq0,$$that is the exponential decay of the first $N$ modes of $z$. The rest of the proof follows immediately as the proof of Theorem \ref{T1}, and it is omitted. $\hfill \Box$

Finally,  for the proof of Theorem \ref{T2}, similarly as for the one of Theorem \ref{T2!}, one may proceed as in \cite[Theorem 4.2]{barbu1}, or as in \cite[Theorem 1.1]{lasi}. Therefore, they are  omitted.

\section{Appendix}

\begin{lemma}\label{l2}Under assumption $(H)$, for any $\rho<\gamma_1<\gamma_2<...<\gamma_N$, we have
\begin{equation}\label{e30}\left|\begin{array}{cccc}\frac{1}{\gamma_1-\lambda_1} & \frac{1}{\gamma_1-\lambda_2} &...& \frac{1}{\gamma_1-\lambda_N}\\
\frac{1}{\gamma_2-\lambda_1} & \frac{1}{\gamma_2-\lambda_2} &...& \frac{1}{\gamma_2-\lambda_N}\\
...&...&...&...\\
\frac{1}{\gamma_N-\lambda_1} & \frac{1}{\gamma_N-\lambda_2} &...& \frac{1}{\gamma_N-\lambda_N} \end{array}\right|\neq0. \end{equation}
\end{lemma}
\begin{proof}
 Let us prove this by mathematical induction over $N$.  Step 1, for $N=2$, we have
$$\left|\begin{array}{cc}\frac{1}{\gamma_1-\lambda_1}&\frac{1}{\gamma_1-\lambda_2}\\
\frac{1}{\gamma_2-\lambda_1}&\frac{1}{\gamma_2-\lambda_2}\end{array}\right|=-\frac{(\lambda_1-\lambda_2)(\gamma_1-\gamma_2)}{(\gamma_1-\lambda_1)(\gamma_2-\lambda_2)(\gamma_2-\lambda_1)(\gamma_1-\lambda_2)}\neq0,$$since $\lambda_1<\lambda_2$ and $\gamma_1<\gamma_2$.

 Step 2, we assume that for $N-1$ the claim is true and prove it for $N$. To this end we have
$$\begin{aligned}&\left|\begin{array}{cccc}\frac{1}{\gamma_1-\lambda_1} & \frac{1}{\gamma_1-\lambda_2} &...& \frac{1}{\gamma_1-\lambda_N}\\
\frac{1}{\gamma_2-\lambda_1} & \frac{1}{\gamma_2-\lambda_2} &...& \frac{1}{\gamma_2-\lambda_N}\\
...&...&...&...\\
\frac{1}{\gamma_{N}-\lambda_1} & \frac{1}{\gamma_{N}-\lambda_2} &...& \frac{1}{\gamma_{N}-\lambda_N}\\ \end{array}\right|=\\&
\text{(Subtracting from the first column the $N^{th}$ one, $...$, from the $(N-1)^{th}$ column the $N^{th}$ one )}\\&
=(-1)^{N-1}\left|\begin{array}{ccccc}\frac{\lambda_{N}-\lambda_1}{(\gamma_1-\lambda_1)(\gamma_1-\lambda_N)}& \frac{\lambda_{N}-\lambda_2}{(\gamma_1-\lambda_2)(\gamma_1-\lambda_N)}&...&\frac{\lambda_{N}-\lambda_{N-1}}{(\gamma_1-\lambda_{N-1})(\gamma_1-\lambda_N)}&\frac{1}{\gamma_1-\lambda_N}\\
\frac{\lambda_{N}-\lambda_1}{(\gamma_2-\lambda_1)(\gamma_2-\lambda_N)}& \frac{\lambda_{N}-\lambda_2}{(\gamma_2-\lambda_2)(\gamma_2-\lambda_N)}&...&\frac{\lambda_{N}-\lambda_{N-1}}{(\gamma_2-\lambda_{N-1})(\gamma_2-\lambda_N)}&\frac{1}{\gamma_2-\lambda_N}\\
...&...&...&...\\
\frac{\lambda_{N}-\lambda_1}{(\gamma_N-\lambda_1)(\gamma_N-\lambda_N)}& \frac{\lambda_{N}-\lambda_2}{(\gamma_N-\lambda_2)(\gamma_N-\lambda_N)}&...&\frac{\lambda_{N}-\lambda_{N-1}}{(\gamma_N-\lambda_{N-1})(\gamma_N-\lambda_N)}&\frac{1}{\gamma_N-\lambda_N}\end{array}\right|\\&
=\frac{(-1)^{N-1}}{\gamma_{N}-\lambda_N}\prod_{k=1}^{N-1}\frac{ \lambda_N-\lambda_k}{\gamma_k-\lambda_N}\left|\begin{array}{ccccc}\frac{1}{\gamma_1-\lambda_1} & \frac{1}{\gamma_1-\lambda_2} &...& \frac{1}{\gamma_1-\lambda_{N-1}}&1\\
\frac{1}{\gamma_2-\lambda_1} & \frac{1}{\gamma_2-\lambda_2} &...& \frac{1}{\gamma_2-\lambda_{N-1}}&1\\
...&...&...&...\\
\frac{1}{\gamma_{N}-\lambda_1} & \frac{1}{\gamma_{N}-\lambda_2} &...& \frac{1}{\gamma_{N}-\lambda_{N-1}}&1\\ \end{array}\right|\\&
\text{(Subtracting the $N^{th}$  line from the first one, ..., the $N^{th}$ line from the $(N-1)^{th}$ one)}\\&
=\frac{(-1)^{N-1}}{\gamma_{N}-\lambda_N}\prod_{k=1}^{N-1}\frac{ (\lambda_N-\lambda_k)(\gamma_N-\gamma_k)}{(\gamma_k-\lambda_N)(\gamma_N-\lambda_k)}\left|\begin{array}{cccc}\frac{1}{\gamma_1-\lambda_1} & \frac{1}{\gamma_1-\lambda_2} &...& \frac{1}{\gamma_1-\lambda_{N-1}}\\
\frac{1}{\gamma_2-\lambda_1} & \frac{1}{\gamma_2-\lambda_2} &...& \frac{1}{\gamma_2-\lambda_{N-1}}\\
...&...&...&...\\
\frac{1}{\gamma_{N-1}-\lambda_1} & \frac{1}{\gamma_{N-1}-\lambda_2} &...& \frac{1}{\gamma_{N-1}-\lambda_{N-1}} \end{array}\right|\neq0
\end{aligned}$$by the inductive hypothesis and the fact that $\left\{\lambda_1\right\}<\lambda_2<...<\lambda_{N}$ and $\gamma_1<\gamma_2<...<\gamma_{N}$.
\end{proof}

\begin{lemma}\label{l1} The sum $B_1+B_2+...+B_N$ is an invertible matrix, where $B_k,\ k=1,...,N$ are introduced in relation (\ref{bo}).
\end{lemma}
\begin{proof}
Arguing by contradiction, let us assume that there is  $z=\left(\begin{array}{c}z_1\\ z_2\\...\\z_N\end{array}\right)\in\mathbb{R}^N$, nonzero, such that $(B_1+...+B_N)z=0.$ It follows that
$$\sum_{k=1}^N\left<B_kz,z\right>_N=0,$$or, equivalently,
$$\sum_{k=1}^N\int_{\Gamma_1}\left(\sum_{i=1}^Nz_i\frac{1}{\gamma_k-\lambda_i}\frac{\partial}{\partial\nu}\phi_i(x)\right)^2dx=0.$$ Because the Lebesgue measure of $\Gamma_1$ is nonzero, we deduce from the above that
$$\sum_{i=1}^Nz_i\frac{1}{\gamma_k-\lambda_i}\frac{\partial}{\partial\nu}\phi_i(x)=0,\  \text{ a.e. on } \Gamma_1,$$ for all $ k=1,...,N.$ This gives  $N\times N$ linear homogeneous systems, with the unknowns $z_i,\ i=1,...,N$, for almost all $x\in \Gamma_1$. By the unique continuation property of the eigenfunctions $\left\{\phi_i\right\}_{i=1}^N$, we know that, for all $i=1,2,...,N$, $\frac{\partial}{\partial\nu}\phi_i\neq 0$ on $\Gamma_1$. The determinant of the matrix of the corresponding system is
$$\begin{aligned}&\left|\begin{array}{cccc}\frac{1}{\gamma_1-\lambda_1}\frac{\partial}{\partial\nu}\phi_1(x) & \frac{1}{\gamma_1-\lambda_2}\frac{\partial}{\partial\nu}\phi_2(x) &...& \frac{1}{\gamma_1-\lambda_N}\frac{\partial}{\partial\nu}\phi_N(x)\\
\frac{1}{\gamma_2-\lambda_1}\frac{\partial}{\partial\nu}\phi_1(x) & \frac{1}{\gamma_2-\lambda_2}\frac{\partial}{\partial\nu}\phi_2(x) &...& \frac{1}{\gamma_2-\lambda_N}\frac{\partial}{\partial\nu}\phi_N(x)\\
...&...&...&...\\
\frac{1}{\gamma_N-\lambda_1}\frac{\partial}{\partial\nu}\phi_1(x) & \frac{1}{\gamma_N-\lambda_2}\frac{\partial}{\partial\nu}\phi_2(x) &...& \frac{1}{\gamma_N-\lambda_N}\frac{\partial}{\partial\nu}\phi_N(x)\end{array}\right|\\&=
\prod_{i=1}^N\frac{\partial}{\partial\nu}\phi_i(x)\left|\begin{array}{cccc}\frac{1}{\gamma_1-\lambda_1} & \frac{1}{\gamma_1-\lambda_2} &...& \frac{1}{\gamma_1-\lambda_N}\\
\frac{1}{\gamma_2-\lambda_1} & \frac{1}{\gamma_2-\lambda_2} &...& \frac{1}{\gamma_2-\lambda_N}\\
...&...&...&...\\
\frac{1}{\gamma_N-\lambda_1} & \frac{1}{\gamma_N-\lambda_2} &...& \frac{1}{\gamma_N-\lambda_N} \end{array}\right|\neq0, \text{ a.e. }x\in\Gamma_1,\end{aligned}$$ by Lemma \ref{l2}. Hence, necessarily $z=0$. This is in contradiction with our assumption. We conclude that the sum $B_1+...+B_N$ is indeed an invertible matrix.
\end{proof}

 \textbf{Acknowledgment} The author is indebted to the anonymous referees for their pertinent and helpful comments that lead to the improvement of the paper. This work was supported by a post-doctoral  fellowship of the Alexander von Humboldt Foundation.

\end{document}